\documentclass[a4paper, 10pt, onecolumn]{article}    

\usepackage{graphicx}
\usepackage{amsfonts}
\usepackage{amssymb}  
\usepackage{amsmath}
\usepackage{amsthm}
\usepackage{latexsym}
\usepackage{hyperref}
\usepackage{color}
\usepackage{balance}
\usepackage{tikz}
\usepackage[utf8]{inputenc}
\usepackage{pgfplots}

\definecolor{color0}{rgb}{0.12156862745098,0.466666666666667,0.705882352941177}
\definecolor{color1}{rgb}{1,0.498039215686275,0.0549019607843137}
\definecolor{color2}{rgb}{0.172549019607843,0.627450980392157,0.172549019607843}

\newcommand{\norm}[1]{\left\lVert#1\right\rVert}

\newcommand{\rr}{\mathbb{R}}

\def\tinf{{\to\infty}}

\newtheorem{assumption}{Assumption}
\newtheorem{lemma}{Lemma}

\newtheorem{theorem}{Theorem}
\newtheorem{corollary}{Corollary}
\newtheorem{idefinition}{Definition}{\rm}
\newenvironment{definition}{\begin{idefinition} \rm}{\end{idefinition}}
\newtheorem{iremark}{Remark}{\rm}
\newenvironment{remark}{\begin{iremark} \rm}{\end{iremark}}

\begin{document}

\title{\bf Distributed interpolatory algorithms\\ for set membership estimation}

\author{
  Francesco Farina, Andrea Garulli, Antonio Giannitrapani
 \thanks{This paper has received funding from the European Union's
     Horizon 2020 Research and Innovation Programme - Societal
     Challenge 1 (DG CONNECT/H) under grant agreement n$^\circ$ 643644
     ``ACANTO - A CyberphysicAl social NeTwOrk using robot
     friends''.}
  \thanks{The authors are with the Dipartimento di Ingegneria dell'Informazione e Scienze Matematiche,
Universit{\`a} di Siena, Siena, Italy. e-mail: {\tt\small\{farina,garulli,giannitrapani\}@diism.unisi.it}}%
} 

\date{}

\maketitle
\begin{abstract}
  This work addresses the distributed estimation problem in a set membership framework. The agents of a network collect measurements which are affected by bounded errors, thus implying that the unknown parameters to be estimated belong to a suitable feasible set. Two distributed algorithms are considered, based on projections of the estimate of each agent onto its local feasible set. The main contribution of the paper is to show that such algorithms are asymptotic interpolatory estimators, i.e. they converge to an element of the global feasible set, under the assumption that the feasible set associated to each measurement is convex. The proposed techniques are demonstrated on a distributed linear regression estimation problem.

\end{abstract}
\section{Introduction}
\label{sec:Intro}
Set membership approaches to estimation problems have been studied since a long time~\cite{S68,BR71}. Under the assumption of bounded noise, the estimation process can be carried out in terms of feasible sets. A feasible measurement set is associated to each measurement. Such a set contains all the possible values of the unknown parameter compatible with the corresponding measurement and the noise bound. The intersection of all the feasible measurement sets gives the feasible parameter set, which is guaranteed to contain the unknown parameter. The set membership paradigm has been applied to a number of estimation problems, ranging from system identification to state estimation for dynamic systems (see, e.g., \cite{N87,CLPR14,cgv14tac} and references therein), including applications in a variety of fields, such as mobile robotics \cite{DGGV03,Cal05} and automotive systems \cite{KSS06,NRM11}. 

A coarse classification of set membership estimation algorithms distinguishes between set-valued estimators, aimed at providing an approximation of the generally intractable true feasible set, and \emph{pointwise estimators}, which return a single estimate of the unknown parameter enjoying some optimality or suboptimality property with respect to a suitable measure of the feasible set itself \cite{MV91}. Among pointwise estimators, the class of \emph{interpolatory estimators} has been widely investigated. These estimators return an element of the feasible set. They possess a number of nice properties and have been studied within different estimation settings (see, e.g., \cite{GK92a,MV93,G99}).

In recent years, the increasing interest toward multi-agent systems has led to the design of distributed versions of set membership estimators in order to exploit the inherent robustness and versatility of the distributed approach. For example, the problem of approximating the feasible parameter set in a sensor network has been addressed by using different approximating regions such as parallelotopes \cite{NF09}, ellipsoids \cite{NPH12} or the union of rectangular cells \cite{LK10}.

In this paper, we focus on the problem of finding an element of the feasible set in a distributed manner, i.e. a \emph{distributed interpolatory estimator}. The number of distributed optimization methods flourished very recently provide a useful framework which can be exploited for distributed estimation as well (e.g., \cite{NO09,NB11,B11} and references therein).
Two different scenarios are considered. In the first one, the agents are active one at a time, in a cyclic order. They proceed by iteratively projecting the current estimate on their local feasible set and then passing the projected estimate to the next active agent. To this aim, each agent  recursively updates its feasible set on the basis of the current measurements. The second scenario assumes that the agents process simultaneously the information provided by the local measurements. After updating the local feasible sets, they perform a consensus step on the current estimates and then project the resulting weighted averages of the estimates on the local feasible sets.
The main feature of both scenarios is that each measurement provides a new feasible measurement set, and therefore asymptotically the procedure has to deal with \emph{infinitely many} sets, a setting for which few theoretical results are available. The convergence properties of the proposed algorithms are analyzed. It is shown that if the measurement sets are convex, in both scenarios the sequence of the estimates converges to a point belonging to the true feasible parameter set, thus providing an asymptotic interpolatory estimator. These results can be seen as an extension of previous results on the convergence of alternating projection methods~\cite{ER11,G13,BB96} and of constrained consensus algorithms~\cite{NOP10,NL17}, in which a finite number of convex sets was involved. 

The paper is structured as follows. In Section~\ref{sec:PbForm}, the estimation problem in the presence of bounded measurement errors is formulated and the two distributed solutions corresponding to the considered scenarios are presented. The convergence properties of the proposed approaches are studied in Section~\ref{sec:convergence-result}. The application of these techniques to a linear regression estimation problem is illustrated in Section~\ref{sec:application}. Finally, some conclusions are drawn in Section~\ref{sec:conclusions}.

\section{Problem formulation}
\label{sec:PbForm}

\subsection{Preliminaries}
\label{sec:preliminaries}
Some basic concepts on the convergence of a sequence of sets are first recalled~\cite{R13}. Let $\{A(k)\}$, $k=1,2,\dots$, be an infinite sequence of subsets of $\rr^n$. The \emph{limit infimum} of the sequence $\{A(k)\}$ is defined as
$$
\lim_{k \tinf} \inf A(k) = \bigcup_{k\geq1} \bigcap_{j\geq k} A(j).
$$
Similarly, the \emph{limit supremum} of the sequence $\{A(k)\}$ is defined as
$$
\lim_{k \tinf} \sup A(k) = \bigcap_{k\geq1} \bigcup_{j\geq k} A(j).
$$
\begin{definition}[Limit of a sequence of sets.]
  The sequence $\{A(k)\}$ converges to the set $A$ if
  $$
  \lim_{k \tinf} \inf A(k) = \lim_{k \tinf} \sup A(k) = A,
  $$ 
  for some $A\subset \rr^n$. \qed
\end{definition}
If sequence $\{A(k)\}$ converges to  $A$, then the set $A$ is called the limit of $\{A(k)\}$ and this is denoted by writing
$$
\lim_{k\tinf} A(k) = A.
$$
\begin{definition}[Nonincreasing sequence of sets.] 
If $A(k+1) \subseteq A(k)$, for all $k$, the sequence $\{A(k)\}$ is called \emph{nonincreasing}. \qed
\end{definition}
It is easy to verify that if $\{A(k)\}$ is nonincreasing, then its limit always exists and it is given by
$$
\lim_{k \tinf}A(k) = \bigcap_{j\geq1} A(j).
$$

\subsection{Set membership estimation}
\label{sec:two-node}
Consider $N$ processing nodes which cooperate to compute an estimate of an unknown parameter $x \in \rr^n$ and denote by $x_i(k)$ the estimate computed at time $k$ by node $i$. Let $\mathcal{G} = (\mathcal{V},\mathcal{E})$, where $\mathcal{V} = \{1,\dots,N\}$ and $\mathcal{E} \subseteq \mathcal{V} \times \mathcal{V}$, be a directed graph describing the communication among the agents. An edge $(j,i) \in \mathcal{E}$ if and only if agent $j$ sends its estimate to agent $i$. In this case, we say that $j$ is a neighbor of $i$. At each time instant $k$, each node $i$ takes a measurement of a known function of $x$, which is corrupted by an unknown-but-bounded (UBB) noise. The UBB assumption allows one to define for each node $i$ and each time instant $k$ a \emph{feasible measurement set}, denoted by $M_i(k)$, which contains all the possible values of $x$ compatible with the current measurement and the noise bound. Moreover, each node stores its \emph{local feasible parameter set}, i.e. the set of the values of $x$ compatible with all the measurements taken by that node \emph{up to} time $k$. Let $X_i(k)$ be the local feasible parameter set for node $i$ at time $k$. Formally,
\begin{equation} \label{eq:Xk}
X_i(k) = \bigcap_{h=1}^k M_i(h), \quad i=1,\dots,N.
\end{equation}
Since $\{X_i(k)\}$ is a nonincreasing sequence of sets by construction, it converges to
\begin{equation} \label{eq:X}
 X_i = \lim_{k \tinf} X_i(k).
\end{equation}
Hence, the \emph{global asymptotic feasible parameter set} can be defined as
\begin{equation}
\label{eq:Xi}
X = \bigcap_{i=1}^N X_i.
\end{equation}

In this work, we consider two different scenarios. In the first one, it is supposed that the agents process their measurements and update their estimates one at a time. For the sake of presentation it will be assumed that in this scenario the communication graph $\mathcal{G}$ is a ring, although the same approach can be applied in a more general setting. In the second scenario, the agents process simultaneously the information locally available and then broadcast the updated estimates to their neighbors. In the following, we present interpolatory estimators for the unknown parameter $x$ in both considered scenarios. The proposed algorithms generate sequences of estimates that eventually converge to a point belonging to the global feasible parameter set.
Let $P_Z[p]$ denote the projection of point $p \in \rr^n$ on the closed set $Z \subset \rr^n$, defined as
$$
P_Z [p] = \arg \min_{z \in Z} \norm{p-z},
$$ 
where $\norm{\cdot}$ denotes the 2-norm in $\rr^n$. 

\subsubsection{Incremental algorithm}
\label{sec:incr-SM}

The estimation of $x$ is carried out by cyclically projecting the current estimate on the local feasible parameter sets of each node. For $i=1,\dots,N$ and $k=0,1,2,\dots$, the incremental estimation algorithm with cyclic order of projections can be written as
 \begin{align}
 X_i(k+1) &= X_i(k) \cap M_i(k+1),   \label{eq:N1}\\
 x_i(k+1) &= P_{X_i(k+1)}[x_{i-1}(k+1)], \label{eq:N3}
 \end{align}
with the initial conditions $x_N(0)=x_0$ and $X_i(0) = \rr^n$, and the convention that $x_0(k+1)= x_N(k)$. 

\subsubsection{Distributed algorithm}
\label{sec:distr-SM}
In this scenario, an agent updates its estimate in two steps. First, it computes a weighted average of its estimate and those of its neighbors. Then, it projects such an average on the current local feasible parameter set. For $i=1,\dots,N$ and $k=0,1,2,\dots$, the distributed estimation algorithm can be written as
 \begin{align}
   X_i(k+1) &= X_i(k) \cap M_i(k+1),   \label{eq:D1}\\
   z_i(k) &= \sum_{j=1}^N a_{ij} x_j(k), \label{eq:D2} \\
   x_i(k+1) &= P_{X_i(k+1)}[z_i(k)], \label{eq:D3}
 \end{align}
with the initial conditions $x_i(0)$ and $X_i(0) = \rr^n$. Clearly, the weights $a_{ij}$ must comply with the topology of the communication graph, i.e. $a_{ij} = 0$ if  $(j,i) \notin \mathcal{E}$. 

\begin{remark}
\label{rem:2}
Notice that the incremental algorithm is not a special case of the distributed algorithm. Indeed, at each time step, in the distributed scheme the agents have to process their measurements simultaneously. In contrast, in the incremental scheme sequential updates of the agent estimates are performed. In this respect, although in~\eqref{eq:N1}-\eqref{eq:N3} a cyclic order of node activation has been assumed, the convergence property of the incremental algorithm that will be presented in Section~\ref{sec:convergence-result1} is preserved also with different activation rules. For instance, it can be shown that the order of the processing agents does not affect the convergence of the algorithm as long as each agent projects infinitely often on its local feasible set. 
\end{remark}

\begin{remark}
\label{rem:1}
The estimation schemes proposed in this paper can be seen as an extension of alternating and distributed projection methods for computing a point lying in the intersection of a \emph{finite} number of sets. In that case, the estimates are projected infinitely many times on each of the considered sets. This is a main difference with the approaches presented here. In fact, the considered scenarios involve the computation of a point lying in the intersection of an \emph{infinite} number of sets, each of which is in general processed just once. This is the typical situation in recursive set membership estimation schemes, in which a new measurement set is generated at each time instant. To the best of our knowledge, very few theoretical results are available for alternating projection schemes on infinite sets~\cite{B65}. The interested reader is referred to \cite{ER11,G13,BB96} for thorough reviews on alternating projection methods. Similarly, the proposed distributed estimation scheme builds on constrained consensus algorithms studied in \cite{NOP10,NL17} and adapts them to the case of infinite sequences of nonincreasing sets.
\end{remark}

In the next section, it will be shown that, under mild assumptions on the feasible measurement sets and on the communication graph, the sequences $\{x_i(k)\}$ generated by both algorithms converge to a common point $\hat x \in X$.

\section{Convergence analysis}
\label{sec:convergence-result}

\subsection{Incremental algorithm}
\label{sec:convergence-result1}

The following assumptions are made in the remaining of the paper.

\begin{assumption}
\label{ass:1}
  The feasible measurement sets $M_i(k)$, $k=1,2,\dots$, are closed convex sets. \qed
\end{assumption}
\begin{assumption}
\label{ass:2}
  The global asymptotic feasible parameter set $X$ is not empty. \qed
\end{assumption}

Under Assumption~\ref{ass:1}, the local feasible parameter sets $X_i(k)$, as well as their limit sets $X_i$, are closed convex sets, being the intersection of (possibly infinitely many) closed convex sets. In the considered framework, Assumption~\ref{ass:2} is always satisfied if the measurement noise does not violate the UBB constraints. The following lemmas are instrumental to prove the convergence of the estimates generated by the incremental algorithm.

\begin{lemma}
  \label{lem:semispace}
  Let $\{Z(k)\}$, $k=1,2,\dots$, be a nonincreasing sequence of closed convex subsets of $\rr^n$ and denote by $Z$ its limit. Consider an arbitrary point $p \in \rr^n$ and let $q(k)=P_{Z(k)}[p]$ be its projection on $Z(k)$. Then, for any $z \in Z$ and $k=1,2,\dots$, it holds 
  \begin{equation}
    \label{eq:z}
    (p-q(k))^{\top}(z-q(k)) \leq 0.
  \end{equation}
\end{lemma}

\begin{proof}
  If $p \in Z(k)$, then $q(k) = p$ and \eqref{eq:z} clearly holds. Conversely, assume $p \notin Z(k)$. Since $Z(k)$ is convex, there exists a halfspace with supporting hyperplane passing through $q(k)$ and orthogonal to $p-q(k)$, that contains the whole set $Z(k)$. Hence, for any $z(k) \in Z(k)$, $(p-q(k))^{\top}(z(k)-q(k)) \leq 0$. The thesis follows by observing that  $\{Z(k)\}$ is a nonincreasing sequence, hence $z \in Z$ implies that $z \in Z(k)$ for all $k$.
\end{proof}

\begin{lemma}
  \label{lem:accpoint}
  Let $\{Z(k)\}$, $k=1,2,\dots$, be a nonincreasing sequence of closed convex subsets of $\rr^n$ and denote by $Z$ its limit. Let $\{ z(k) \in Z(k) \}$ be a sequence admitting a convergent subsequence $\{z(k_j)\}$, $j=1,2,\dots,$ to a point $\hat z$. Then, $\hat z \in Z$.
\end{lemma}
\begin{proof}
By contradiction, assume that $\hat z \notin Z$ and let $d$ be the distance of $\hat z$ to $Z$. Notice that since $Z$ is closed, $d$ must be strictly greater than zero. From the assumption of convergence, for any $\varepsilon > 0$  there exists a ball of radius $\varepsilon$ and centered at $\hat z$ which contains the subsequence $\{ z(k_j) \}$ from some $\bar j$ onward. Since $z(k_j) \in Z(k_j)$ by assumption, it follows that 
\begin{equation}
  \label{eq:contr}
  Z(k_j) \cap \mathcal{B}(\hat z,\varepsilon) \neq \emptyset, \quad \forall j \geq \bar j.
\end{equation}
But, if we take $\varepsilon < d$, then $Z \cap \mathcal{B}(\hat z,\varepsilon) = \emptyset$, which implies that there exists a $\bar k$ such that $Z(k) \cap \mathcal{B}(\hat z,\varepsilon) = \emptyset$ for all $k \geq \bar k$, since the sequence $\{ Z(k)\}$ is nonincreasing. This clearly 
contradicts \eqref{eq:contr}.  
\end{proof}

The main convergence result can now be proved by resorting to arguments adapted from the case of finite sets (e.g., see \cite{BD03}).
\begin{theorem}
  \label{th:covergence}
  Let Assumptions \ref{ass:1} and \ref{ass:2} hold. Let $\{x_i(k)\}$, $i=1,\dots,N$, $k=1,2,\dots$, be the sequences of estimates computed according to~\eqref{eq:N1},\eqref{eq:N3}. Then, there exists a point $\hat x$ such that
$$
\lim_{k\tinf} x_i(k) = \hat x \in X,\quad i=1,\dots,N,
$$
where $X$ is given by \eqref{eq:Xk}-\eqref{eq:Xi}.
\end{theorem}
\begin{proof}
Let $\bar x$ be any point in $X$. From \eqref{eq:N3} one has
\begin{align*}
  \norm{x_{i-1}(k) - \bar x}^2 &= \norm{x_{i-1}(k) - x_{i}(k) + x_{i}(k) - \bar x}^2 \\
&= \norm{x_{i-1}(k) - x_i(k)  }^2 +  \norm{x_i(k)-\bar x}^2 \\ 
& \phantom{=} + 2(x_{i-1}(k) - x_i(k))^{\top}(x_i(k)-\bar x) \\
  & \geq \norm{ x_{i-1}(k) - x_i(k) }^2 +  \norm{x_i(k)-\bar x}^2,
\end{align*}
where the inequality follows from Lemma~\ref{lem:semispace}.
Hence, 
\begin{equation}
  \label{eq:yxi}
  \norm{x_i(k)-\bar x}^2 \leq \norm{x_{i-1}(k) - \bar x}^2 - \norm{x_{i-1}(k) - x_i(k) }^2,
\end{equation}
which implies that the sequence
$$
\{ \norm{x_1(1)-\bar x}^2,\dots, \norm{x_N(1)-\bar x}^2,\ \norm{x_1(2)-\bar x}^2,\dots \}
$$
is bounded and nonincreasing. Hence, it converges to some $\bar d \geq 0$. Consequently, the subsequences $\{ \norm{x_i(k)-\bar x}^2 \}$ converge to $\bar d$ as well
and, from \eqref{eq:yxi}, 
\begin{equation}
  \label{eq:xydist}
  \lim_{k \tinf} \norm{x_i(k)- x_j(k) }^2 = 0, \quad i,j=1,\dots,N.
\end{equation}
Since $\{\norm{x_i(k) - \bar x}^2\}$ is bounded, so are the estimates $x_i(k)$. Hence, from~\eqref{eq:xydist}, the sequences $\{x_i(k)\}$ admit subsequences $\{x_i(k_h)\}$, $h=1,2,\dots,$ converging to a common point $\hat x$. By invoking Lemma~\ref{lem:accpoint} for the sequences $\{x_i(k)\}$, it turns out that $\hat x$ must belong to $X_i$, $i=1,\dots,N$, and hence $\hat x \in X$. We can now replace $\bar x$ in \eqref{eq:yxi} with $\hat x$  and observe that $\{\norm{x_i(k)- \hat x}^2\}$ are nonincreasing bounded sequences, and hence they converge. But, since it has been previously shown that the subsequences $\{\norm{x_i(k_h)- \hat x}^2\}$ converge to zero, then necessarily
$$
\lim_{k\tinf} \norm{x_i(k)- \hat x}^2 = 0, \quad i=1,\dots,N,
$$
which concludes the proof.
\end{proof}

\begin{corollary}
  \label{cor:conv}
    Let Assumptions \ref{ass:1} and \ref{ass:2} hold. Let $\{x_i(k)\}$, $i=1,\dots,N$, $k=1,2,\dots$, be the sequences of estimates computed according to~\eqref{eq:N1},\eqref{eq:N3}. If $X = \{ x \}$, then 
$$
\lim_{k\tinf} x_i(k) = x, \quad i=1,\dots,N.
$$ \qed
\end{corollary}

Corollary~\ref{cor:conv} states that if the asymptotic feasible parameter set is a singleton, then the sequence of the estimates converges to the true value $x$.

\subsection{Distributed algorithm}
The following assumptions are made, in order to guarantee the convergence of the consensus protocol embedded in the distributed algorithm~\eqref{eq:D1}-\eqref{eq:D3}.
\begin{assumption}
  \label{ass:3}
  The communication graph $\mathcal{G}$ is strongly connected. \qed
\end{assumption}

\begin{assumption}
  \label{ass:4}
  For all $i,j=1,\dots,N$, the weights $a_{ij} \geq 0$ in \eqref{eq:D2} satisfy
  \begin{itemize}
  \item $a_{ii} > 0$;
  \item if $i \neq j$, $a_{ij} > 0$ if and only if $(j,i) \in \mathcal{E}$;
  \item $\sum_{j=1}^N a_{ij}=1$.    \qed
  \end{itemize} 
\end{assumption}
The following result holds (see, e.g., Theorems 7.6 and 7.10 in \cite{CL08}).
\begin{lemma}
  \label{lem:markov}
  Let Assumptions~\ref{ass:3} and \ref{ass:4} hold. Then, there exists a unique $v \in \rr^N$, with $v_i> 0$, $i=1,\dots, N$, $\sum_{i=1}^N v_i = 1$ and such that $\sum_{i=1}^N v_i a_{ij} = v_j$, $j=1,\dots,N$. \qed
\end{lemma}

By exploiting the input-to-state stability property of consensus protocols \cite{KRB05}, the following result holds.
\begin{lemma}
  \label{lem:ISS}
  Let Assumptions~\ref{ass:3} and \ref{ass:4} hold and consider the dynamic system
  \begin{equation}
    \label{eq:aver}
      x_i(k+1) = \sum_{j=1}^N a_{ij} x_j(k) + u_i(k), \quad i=1,\dots,N,
  \end{equation}
where $x_i(k) \in \rr^n$, $u_i(k) \in \rr^n$. If $\lim_{k\tinf} u_i(k) = 0$, for $i=1,\dots,N$, then $\lim_{k \tinf} \norm{x_i(k) - x_j(k)} = 0$, for all $i,j$.~\qed
\end{lemma}

\begin{theorem}
  \label{th:covergenceD}
  Let Assumptions \ref{ass:1}-\ref{ass:4} hold. Let $\{x_i(k)\}$, $i=1,\dots,N$, $k=1,2,\dots$, be the sequences of estimates computed according to~\eqref{eq:D1}-\eqref{eq:D3}. Then, there exists a point $\hat x$ such that
$$
\lim_{k\tinf} x_i(k) = \hat x \in X,\quad i=1,\dots,N,
$$
where $X$ is given by \eqref{eq:Xk}-\eqref{eq:Xi}.
\end{theorem}
\begin{proof}
Let $\bar x$ be any point in $X$. From \eqref{eq:D2},\eqref{eq:D3}, one has that for $i=1,\dots,N$
\begin{equation}
  \label{eq:z1}
  \norm{x_i(k+1) -\bar x }^2 \leq  \norm{z_i(k) -\bar x }^2 \leq \sum_{j=1}^N a_{ij} \norm{x_j(k) - \bar x}^2,
\end{equation}
where the first inequality comes from the properties of projections on convex sets and the second one from Jensen's inequality. Moreover, similarly to the proof of Theorem~\ref{th:covergence}, from Lemma~\ref{lem:semispace} it follows that
\begin{equation}
  \label{eq:z2}
  \norm{x_i(k+1) - \bar x }^2 \leq \norm{z_i(k) - \bar x }^2 - \norm{x_i(k+1) - z_i(k) }^2.
\end{equation}
Let the vector $v\in\rr^N$ be defined as in Lemma~\ref{lem:markov}. From~\eqref{eq:z2} one gets
\begin{equation}
  \label{eq:z3}
  \begin{split}
  \sum_{i=1}^N v_i \norm{x_i(k+1) - \bar x }^2 \leq \sum_{i=1}^N v_i  \norm{z_i(k) - \bar x }^2 \\- \sum_{i=1}^N v_i  \norm{x_i(k+1) - z_i(k) }^2.
\end{split}
\end{equation}
Moreover, inequalities~\eqref{eq:z1} imply that
\begin{align}
  \sum_{i=1}^N v_i \norm{x_i(k+1) -\bar x }^2 &\leq \sum_{i=1}^N v_i  \norm{z_i(k) -\bar x }^2  \nonumber\\
  &\leq \sum_{i=1}^N v_i  \sum_{j=1}^N a_{ij} \norm{x_j(k) - \bar x}^2    \label{eq:d1}\\
  & = \sum_{i=1}^N v_i  \norm{x_i(k) - \bar x}^2, \nonumber
\end{align}
where the last equality stems from Lemma~\ref{lem:markov}. By \eqref{eq:d1}, for $i=1,\dots,N$, the sequence
\begin{equation} \nonumber
\begin{split}
\{\sum_{i=1}^N \!v_i \!\norm{x_i(1) \!-\!\bar x }^2\!, \sum_{i=1}^N \!v_i\!  \norm{z_i(1)\! -\!\bar x }^2\!, \sum_{i=1}^N\! v_i\! \norm{x_i(2)\! -\!\bar x }^2\!, \! \dots \}
\end{split}
\end{equation}
is bounded and nonincreasing. Hence, it converges to some $\bar d \geq 0$, i.e.
\begin{equation}
  \label{eq:zconv}
  \lim_{k\tinf} \sum_{i=1}^N v_i  \norm{x_i(k) - \bar x }^2 = \lim_{k\tinf} \sum_{i=1}^N v_i  \norm{z_i(k) - \bar x }^2 = \bar d.
\end{equation}
From~\eqref{eq:z3} and \eqref{eq:zconv}, it follows
$$
\lim_{k\tinf} \sum_{i=1}^N v_i  \norm{x_i(k+1) - z_i(k) }^2 = 0,
$$
which, together with Lemma~\ref{lem:markov}, implies that 
\begin{equation}
  \label{eq:z4}
  \lim_{k\tinf} \norm{x_i(k+1) - z_i(k) }^2 = 0.
\end{equation}
Now, observe that \eqref{eq:D2},\eqref{eq:D3} can be rewritten as~\eqref{eq:aver}, where 
\begin{equation}
  \label{eq:ueq}
  u_i(k) = P_{X_i(k+1)}[z_i(k)] - z_i(k).
\end{equation}
Since \eqref{eq:D3}, \eqref{eq:z4} and \eqref{eq:ueq} imply that $\lim_{k\tinf} u_i(k) = 0$, from Lemma~\ref{lem:ISS} one gets
\begin{equation}
  \label{eq:conv}
  \lim_{k\tinf} \norm{x_i(k)- x_j(k)} = 0, \quad i,j=1,\dots,N.
\end{equation}
Being $\{ \sum_{i=1}^N v_i \norm{x_i(k) -\bar x }^2 \}$ bounded, so are the sequences $\{x_i(k)\}$. Hence, from \eqref{eq:conv}, there exist subsequences $\{x_i(k_h)\}$ converging to a common $\hat x$, which implies
\begin{equation}
  \label{eq:convx2}
  \lim_{h\tinf} \sum_{i=1}^N v_i \norm{x_i(k_h) - \hat x}^2 = 0.
\end{equation}
From Lemma~\ref{lem:accpoint}, $\hat x \in X$. By setting $\bar x = \hat x$ in~\eqref{eq:zconv} and exploiting \eqref{eq:convx2}, we conclude that
$$
\lim_{k\tinf} \sum_{i=1}^N v_i \norm{x_i(k) - \hat x}^2 = 0, 
$$
which, from Lemma \ref{lem:markov}, implies that $\lim_{k\tinf} x_i(k) = \hat x$, $i=1,\dots,N$.
\end{proof}

\begin{corollary}
  \label{cor:conv2}
  Let Assumptions \ref{ass:1}-\ref{ass:4} hold. Let $\{x_i(k)\}$, $i=1,\dots,N$, $k=1,2,\dots$, be the sequences of estimates computed according to~\eqref{eq:D1}-\eqref{eq:D3}. If $X = \{ x \}$, then 
$$
\lim_{k\tinf} x_i(k) = x.
$$ \qed
\end{corollary}

\begin{remark}
\label{rem:3}
Both estimation algorithms presented in this paper require a very small amount of information to be exchanged among the agents. In fact, only the current estimates are shared in the network, whereas both the measurements and the corresponding feasible sets are computed and stored locally by each agent.
\end{remark}

\begin{remark}
\label{rem:4}
Any choice of the weights satisfying Assumption~\ref{ass:4} guarantees convergence of the estimates. Possible alternatives falling into this class are uniform weights, Metropolis-Hastings weights or maximum degree weights. Together with the graph topology, the choice of the weights affects the convergence rate (e.g., see \cite{XB04}).
\end{remark}

\section{Application to linear regression models}
\label{sec:application}
The nonincreasing property of the sequence of sets on which the estimates are projected is crucial for the convergence of the proposed algorithms to the global asymptotic feasible parameter set $X$. Indeed, when such a property is not satisfied, counterexamples can be easily constructed.

The only requirement on the measurement sets $M_i(k)$ for the above results to hold is that they be closed convex sets (not even necessarily bounded). In practice, the proposed approaches can be applied whenever the intersection $X_i(k)$ of the measurement sets can be computed and stored easily (see~\eqref{eq:N1} and~\eqref{eq:D1}). This actually occurs in many situations of practical interest, in which the measurement sets have some suitable structure, such as strips, circular sectors, parallelotopes or, more generally, weighted balls in some vector norm. In the following, we use the proposed algorithms to estimate the parameters of a linear regression model.

Consider a sensor network composed by $N$ nodes. At time $k$, each node $i$ takes a measurement of the unknown parameter $\theta^\star\in\rr^n$ according to the  model
\begin{equation}\label{eq:measStrip}
  y_i(k)=\varphi_i^{\top} \theta^\star+w_i(k),
\end{equation}
where $y_i(k)\in\rr$, $\varphi_i\in\rr^n$ and $w_i(k)\in\rr$ is an UBB noise such that
\begin{equation}\label{eq:UBB}
  |w_i(k)|\leq \epsilon_i, \quad \forall k,
\end{equation}
where $\epsilon_i\geq 0$. Without loss of generality, we assume the regressors $\varphi_i$ to be normalized, i.e. $\|\varphi_i\|=1$, for $i=1,\dots,N$. From \eqref{eq:measStrip} and~\eqref{eq:UBB}, we have that each measurement set is a strip in the parameter space, i.e.
\begin{equation}
  \label{eq:MkiStrip}
  M_i(k)=\{\theta \in \rr^n : |\varphi_i^{\top} \theta-y_i(k)|\leq \epsilon_i\}.
\end{equation}
Notice that, from~\eqref{eq:N1} and~\eqref{eq:MkiStrip}, the local feasible set for node $i$ at time $k$ can be written as
\begin{equation}
\label{eq:XkiStrip}
  X_i(k)=\{\theta \in \rr^n : \max_{h\leq k}\{y_i(h)\}-\epsilon_i \leq \varphi_i^{\top} \theta\leq \min_{h\leq k}\{y_i(h)\}+\epsilon_i \}.
\end{equation}
We stress that~\eqref{eq:XkiStrip} can be easily computed by each node by storing just the maximum and minimum value of the measurements $y_i(k)$. 

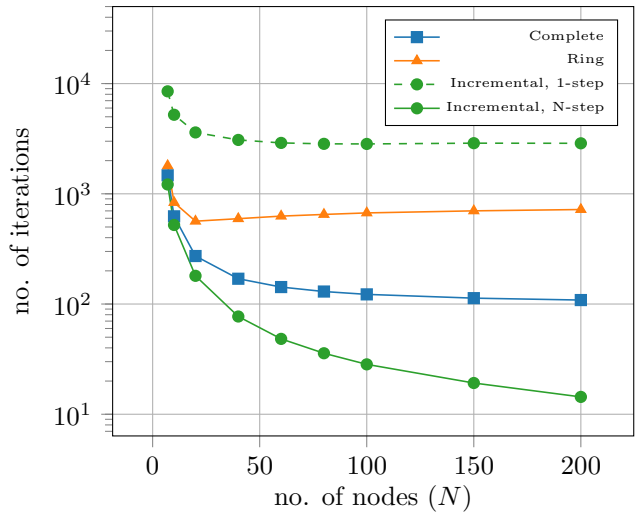
\begin{figure}
  \centering
  \begin{tikzpicture}

	\definecolor{color0}{rgb}{0.12156862745098,0.466666666666667,0.705882352941177}
	\definecolor{color1}{rgb}{1,0.498039215686275,0.0549019607843137}
	\definecolor{color2}{rgb}{0.172549019607843,0.627450980392157,0.172549019607843}
	
	\begin{axis}[
	title={},
	xmin=-18.7, xmax=224.7,
	ymin=-173.9138, ymax=50000,
	tick align=outside,
	tick pos=left,
	ymode = log,
	xlabel = {no. of nodes ($N$)},
	ylabel = {no. of iterations},
	xmajorgrids,
	x grid style={white!69.019607843137251!black},
	ymajorgrids,
	y grid style={white!69.019607843137251!black},
	legend style={
		font = \tiny,
	        cells={anchor=east},
	        legend pos= north east,
	    }
	]
	\addplot [semithick, color0, mark=square*, mark size=2, mark options={solid}]
	table {%
	7 1472.664
	10 624.937999999999
	20 272.636
	40 169.639
	60 142.553
	80 129.828
	100 122.356
	150 113.069
	200 108.666
	};
	\addlegendentry{Complete}
	
	\addplot [semithick, color1, mark=triangle*, mark size=2, mark options={solid}]
	table {%
	7 1799.242
	10 834.196999999999
	20 564.625999999999
	40 594.842999999999
	60 627.551
	80 649.988000000001
	100 671.092
	150 700.54
	200 720.306999999999
	};
	\addlegendentry{Ring}
	
	\addplot [semithick, dashed, color2, mark=*, mark size=2, mark options={solid}]
	table {%
	7 8509.90800000001
	10 5214.99799999999
	20 3605.02
	40 3082.506
	60 2896.898
	80 2840.873
	100 2835.374
	150 2880.771
	200 2873.071
	};
	\addlegendentry{Incremental, 1-step}
	
	\addplot [semithick, color2, mark=*, mark size=2, mark options={solid}]
	table {%
	7 1215.70114285714
	10 521.499799999999
	20 180.251
	40 77.06265
	60 48.2816333333333
	80 35.7609125
	100 28.35374
	150 19.20514
	200 14.365355
	};
	\addlegendentry{Incremental, N-step}
	
	\end{axis}
	
	\end{tikzpicture}
  \caption{Average number of iterations for different number of nodes $N$.}
  \label{fig:logiter}
\end{figure}

In order to compare the convergence rate of the two algorithms, a Monte Carlo simulation has been performed for the case $n=5$. For different values of $N$ ranging from $7$ to $200$, 1000 simulation runs have been performed. At each run, the initial estimates $x_i(0)$ are uniformly generated in $[-5,\, 5]^n$. Similarly, the regressors $\varphi_i$ in~\eqref{eq:measStrip} are uniformly generated in $[0,\, 1]^n$ and then normalized so that $\|\varphi_i\|=1$. The measurement errors $w_i(k)$ are uniformly distributed in $[-\epsilon_i,\ \epsilon_i]$, with  $\epsilon_i$ selected at random in $[0.10,\ 0.13]$. The distributed algorithm has been tested with two differpsent network topologies, namely a complete graph and a ring. The weights have been chosen so that each agent computes the average of its estimate and those of its neighbors, i.e. the non-zero weights are set to $1/N$ for the complete graph and to $1/3$ for the ring graph. Each simulation run is stopped when the distance of the estimates of all the nodes to $X$ is smaller than $\delta=10^{-3}$. The true value of the unknown parameter $\theta^\star$ is uniformly generated in $[-5,\, 5]^n$.

The average number of iterations needed to reach the given accuracy $\delta$ is shown in Fig.~\ref{fig:logiter}. It can be seen that the incremental estimation scheme (green solid line) features a faster convergence rate than the distributed ones, the difference increasing with the number of nodes. As expected, among the two network topologies considered, the complete graph (blue solid line) shows a faster convergence with respect to the ring graph (orange solid line). However, in the above comparison it is implicitly assumed that the incremental algorithm is able to carry out a complete cycle of node activations between two consecutive measurements (i.e. the whole cycle of $N$ sequential projections can be completed before a new set of measurements is available). This may not be the case in many real-world scenarios (e.g. large networks, nodes with limited processing power, reduced communication bandwidth). In the limit case in which the time required by one node for processing a measurement and passing its estimate to its neighbors is comparable with the sampling time of the measurement process, the incremental algorithm as defined in \eqref{eq:N1}-\eqref{eq:N3} must be modified in order to account for the fact that only one node can be active at each time $k$. In this scenario, the distributed scheme is clearly faster than the incremental one (green dashed line), even in case of scarcely connected topologies like a ring.

\section{Conclusions}
\label{sec:conclusions}
Two interpolatory algorithms for distributed set membership estimation have been proposed. Under the assumption of convex measurement sets, both techniques generate  sequences of estimates that converge to a point lying in the global asymptotic feasible set. The main technical contribution of this work is the proof of convergence of the proposed estimation schemes. This is a non trivial result, since projections on infinitely many sets are involved. From a practical viewpoint, the proposed techniques are general enough to be suitable for any kind of sensors whose uncertainty measurement set is convex. Numerical results show that the incremental estimation scheme provides better results in terms of convergence rate than its distributed counterpart. However, a number of factors may limit its applicability (e.g.,  large networks, slow processors or communication link, high measurement frequency). In all this cases, the distributed algorithm represents an effective alternative.

Several extensions of this work are currently under investigation, including studying the convergence rate of the proposed algorithms, extending the convergence analysis to time-varying communication graphs, devising distributed algorithms for set-valued estimation. A further problem of interest when the network is a design parameter is to find the combination of graph topology and edge weights maximizing the convergence rate.

\balance
\bibliographystyle{IEEEtran}
\bibliography{biblio}

\end{document}